\documentclass[12pt]{article}
\usepackage{pifont}

\usepackage{amsfonts}
\usepackage{latexsym}
\usepackage{amsmath}
\usepackage{amssymb}
\usepackage{color}

\newtheorem{theorem}{Theorem}[section]

\newtheorem{proposition}[theorem]{Proposition}

\newtheorem{example}[theorem]{Example}

\newenvironment{proof}[1][Proof]{\noindent \textbf{#1.} }{\ \ \  $\Box$}

\newtheorem{remark}[theorem]{Remark}

\numberwithin{equation}{section}

\title{Anticipated backward doubly stochastic differential equations}

\date{}

 \author{ Xiaoming Xu\thanks{E-mail: xmxu@njnu.edu.cn}
 \\ \small{Institute of Finance and Statistics, School of Mathematical Sciences,}
 \\ \small{Nanjing Normal University, Nanjing, 210023, China}
 }

\begin{document}

\maketitle

\begin{abstract}
In this paper, we deal with a new type of differential equations
called anticipated backward doubly stochastic differential equations
(anticipated BDSDEs). The coefficients of these BDSDEs depend on the
future value of the solution $(Y, Z)$. We obtain the existence and
uniqueness theorem and a comparison theorem for the solutions of
these equations. Besides, as an application, we also establish a
duality between the anticipated BDSDEs and the delayed doubly
stochastic differential equations (delayed DSDEs).
\\
\par $\textit{Keywords:}$ anticipated backward doubly stochastic
differential equation, comparison theorem, duality
\end{abstract}



\section{Introduction}

Backward stochastic differential equation (BSDE) was considered the
general form the first time by Pardoux-Peng \cite{PP1} in 1990. In
the last twenty years, the theory of BSDEs has been studied with
great interest due to its applications in the pricing/hedging
problem (see e.g. \cite{KEM1997, KPQ}), in the stochastic control
and game theory (see e.g. \cite{KPQ, HL1}), and in the theory of
partial differential equations (see e.g. \cite{BBP1997, BQR, PP2}).

In order to give a probabilistic representation for a class of
quasilinear stochastic partial differential equations (SPDEs),
Pardoux-Peng \cite{PP3} first studied the backward doubly stochastic
differential equations (BDSDEs) of the general form
\begin{equation}\label{equation:PP BDSDE}
Y_t=\xi+\int_t^T f(s, Y_s, Z_s)ds +\int_t^T g(s, Y_s,
Z_s)d\overleftarrow{B}_s - \int_t^T Z_s d{W}_s,\ \ \ \ t \in [0, T],
\end{equation}
where the integral with respect to $\{B_t\}$ is a "backward It\^{o}
integral", and the integral with respect to $\{W_t\}$ is a standard
forward integral. Note that these two types of integrals are
particular cases of the It\^{o}-Skorohod integral, see
Nualart-Pardoux \cite{NP}. Pardoux-Peng \cite{PP3} proved that under
Lipschitz condition on the coefficients, BDSDE (\ref{equation:PP
BDSDE}) has a unique solution. Since then, the theory of BDSDEs has
been developed rapidly by many researchers. Bally-Matoussi \cite{BM}
gave the probabilistic representation of the solutions in Sobolev
space of semilinear SPDEs in terms of BDSDEs. Matoussi-Scheutzow
\cite{MS} studied BDSDEs and their applications in SPDEs. Shi et al.
\cite{SGL} proved a comparison theorem for BDSDEs with Lipschitz
condition on the coefficients. Lin \cite{Lin} obtained a generalized
comparison theorem and a generalized existence theorem of BDSDEs.

On the other hand, recently, Peng-Yang \cite{PY} (see also \cite{Y})
introduced the socalled anticipated BSDEs (ABSDEs) of the following
form:
\begin{equation*}
\left\{
\begin{tabular}{rlll}
$-dY_t$ &=& $f(t, Y_t, Z_t, Y_{t+\delta(t)},
Z_{t+\zeta(t)})dt-Z_tdW_t, $ & $
t\in[0, T];$\\
$Y_t$ &=& $\xi_t, $ & $t\in[T, T+K];$\\
$Z_t$ &=& $\eta_t, $ & $t\in[T, T+K],$
\end{tabular}\right.
\end{equation*}
where $\delta(\cdot): [0, T]\rightarrow \mathbb{R^+} \setminus
\{0\}$ and $\zeta(\cdot): [0, T]\rightarrow \mathbb{R^+} \setminus
\{0\}$ are continuous functions satisfying

$\mathbf{(a1)}$ there exists a constant $K \geq 0$ such that for
each $t\in[0, T],$
$$t+\delta(t) \leq T+K,\quad t+\zeta(t) \leq T+K;$$

$\mathbf{(a2)}$ there exists a constant $M \geq 0$ such that for
each $t\in[0, T]$ and each nonnegative integrable function
$g(\cdot)$, $$\int_t^T g(s+\delta(s))ds\leq M\int_t^{T+K}
g(s)ds,\quad \int_t^T g(s+\zeta(s))ds\leq M\int_t^{T+K} g(s)ds.$$
Peng-Yang \cite{PY} proved the existence and uniqueness of the
solution to the above equation, and studied the duality between
anticipated BSDEs and delayed SDEs.

In this paper, we are interested in the following BDSDEs with
coefficients depending on the future value of the solution $(Y, Z)$:
\begin{equation}\label{equation:main}
\left\{
\begin{tabular}{rlll}
$-dY_t$ &=& $f(t, Y_t, Z_t, Y_{t+\delta(t)}, Z_{t+\zeta(t)})dt$ &
\vspace{2mm}
\\ && $+ g(t, Y_t, Z_t, Y_{t+\delta(t)}, Z_{t+\zeta(t)}) d\overleftarrow{B}_t - Z_t d{W}_t, $
& $ t\in[0, T];$ \vspace{2mm}
\\
$Y_t$ &=& $\xi_t, $ & $t\in[T, T+K];$ \vspace{2mm}
\\
$Z_t$ &=& $\eta_t, $ & $t\in[T, T+K],$
\end{tabular}\right.
\end{equation}
where $\delta >0 $ and $\zeta >0$ satisfy $(a1)$-$(a2)$.

We prove that under proper assumptions, the solution of the above
anticipated BDSDE (ABDSDE) exists uniquely, and a comparison theorem
is given for the $1$-dimensional anticipated BDSDEs. It may be
mentioned here that, to deal with (\ref{equation:main}), the most
important thing for us is to establish the similar conclusions as in
\cite{PP3} and \cite{SGL} for BDSDE (\ref{equation:PP BDSDE}) with
$\xi$ belonging to a larger space. Besides, as an application, we
study a duality between the anticipated BDSDE and delayed DSDE.

The paper is organized as follows: in Section $2$, we make some
preliminaries. In Section $3$, we mainly study the existence and
uniqueness of the solutions of anticipated BDSDEs, and in Section
$4$, a comparison result is given. As an application, in Section
$5$, we establish a duality between an anticipated BDSDE and a
delayed DSDE. Finally in Section $6$, the conclusion and future work
are presented.

\section{Preliminaries}

Let $T > 0$ be fixed throughout this paper. Let $\{W_t\}_{t\in [0,
T]}$ and $\{B_t\}_{t\in [0, T]}$ be two mutually independent
standard Brownian motion processes, with values respectively in
$\mathbb{R}^d$ and $\mathbb{R}^l$, defined on a probability space
$(\Omega, \mathcal{F}, P)$. Let $\mathcal{N}$ denote the class of
$P-$null sets of $\mathcal{F}$. We define
$$\mathcal{F}_t := \mathcal{F}_{0, t}^W \vee \mathcal{F}_{t, T}^B,\
t\in [0, T];\ \ \mathcal{G}_s := \mathcal{F}_{0, s}^W \vee
\mathcal{F}_{s, T+K}^B,\ s\in [0, T + K],$$ where for any processes
$\{\varphi_t\}$,
$\mathcal{F}_{s,t}^\varphi=\sigma\{\varphi_r-\varphi_s, s \leq r
\leq t\} \vee \mathcal{N}$. We will use the following notations:
\begin{itemize}
\item{$L^2(\mathcal{G}_T; \mathbb{R}^m)$ := $\{\xi\in \mathbb{R}^m$
$|$ $\xi$ is a $\mathcal{G}_T$-measurable random variable such that
$E|\xi|^2< + \infty\};$}

\item{$L_{\mathcal{G}}^2(0, T; \mathbb{R}^m)$ := $\{ \varphi:
\Omega\times [0, T]\rightarrow \mathbb{R}^m$ $|$ $\varphi$ is a
$\mathcal{G}_t$-progressively measurable process such that
$E\int_0^T |\varphi_t|^2dt< + \infty\};$}

\item{$S_{\mathcal{G}}^2(0, T; \mathbb{R}^m)$ := $\{\varphi: \Omega\times
[0, T]\rightarrow \mathbb{R}^m$ $|$ $\varphi$ is a continuous and
$\mathcal{G}_t$-progressively measurable process such that
$E[\sup_{0 \leq t \leq T} |\varphi_t|^2]< + \infty\}.$}
\end{itemize}

\begin{remark}
It should be mentioned here that, the existing result about BDSDEs
are established almost under the condition that the terminal value
$\xi$ is $\mathcal{F}_T$-measurable (see \cite{PP3}, \cite{SGL},
etc.). In this paper, we will first treat the case when $\xi$ is
$\mathcal{G}_T$-measurable.
\end{remark}

\noindent For each $t\in [0, T]$, let
$$f(t, \cdot, \cdot, \cdot, \cdot): \Omega \times [0, T]\times \mathbb{R}^m\times
\mathbb{R}^{m\times d}\times L_{\mathcal{G}}^2(t, T+K;
\mathbb{R}^m)\times L_{\mathcal{G}}^2(t, T+K; \mathbb{R}^{m\times
d})\rightarrow L^2 (\mathcal{G}_t; \mathbb{R}^m),$$ $$g(t, \cdot,
\cdot, \cdot, \cdot): \Omega \times [0, T]\times \mathbb{R}^m\times
\mathbb{R}^{m\times d}\times L_{\mathcal{G}}^2(t, T+K;
\mathbb{R}^m)\times L_{\mathcal{G}}^2(t, T+K; \mathbb{R}^{m\times
d})\rightarrow L^2 (\mathcal{G}_t; \mathbb{R}^{m\times l}).$$ We
make the following hypotheses:

${\bf{(H1)}}$ There exists a constant $c > 0$ such that for any $r,
\bar{r} \in [t, T+K]$, $(t, y, z, \theta, \phi)$, $(t, y^\prime,
z^\prime, \theta^\prime, \phi^\prime)\in [0, T] \times \mathbb{R}^m
\times \mathbb{R}^{m\times d} \times L_{\mathcal{G}}^2(t, T+K;
\mathbb{R}^m)\times L_{\mathcal{G}}^2(t, T+K; \mathbb{R}^{m\times
d})$,
\begin{equation*}
|f(t, y, z, \theta_r, \phi_{\bar{r}})-f(t, y^\prime, z^\prime,
\theta_r^\prime, \phi_{\bar{r}}^\prime)|^2 \leq c(|y-
y^\prime|^2+|z-z^\prime|^2+E^{\mathcal{F}_t}[|\theta_r-\theta_r^\prime|^2+|\phi_{\bar{r}}-\phi_{\bar{r}}^\prime|^2]).
\end{equation*}

${\bf{(H2)}}$ $E[\int_0^T |f(s, 0, 0, 0, 0)|^2ds]< + \infty.$

${\bf{(H3)}}$ There exist  constants $c > 0$, $0 < \alpha_1 < 1$, $0
\leq \alpha_2 < \frac{1}{M}$, satisfying $0 < \alpha_1 + \alpha_2 M<
1$, such that for any $r, \bar{r} \in [t, T+K]$, $(t, y, z, \theta,
\phi)$, $(t, y^\prime, z^\prime, \theta^\prime, \phi^\prime)\in [0,
T] \times \mathbb{R}^m \times \mathbb{R}^{m\times d} \times
L_{\mathcal{G}}^2(t, T+K; \mathbb{R}^m)\times L_{\mathcal{G}}^2(t,
T+K; \mathbb{R}^{m\times d})$,
\begin{equation*}
|g(t, y, z, \theta_r, \phi_{\bar{r}})-g(t, y^\prime, z^\prime,
\theta_r^\prime, \phi_{\bar{r}}^\prime)|^2 \leq c(|y-
y^\prime|^2+E^{\mathcal{F}_t}|\theta_r-\theta_r^\prime|^2)+ \alpha_1
|z-z^\prime|^2+ \alpha_2
E^{\mathcal{F}_t}|\phi_{\bar{r}}-\phi_{\bar{r}}^\prime|^2.
\end{equation*}

${\bf{(H4)}}$ $E[\int_0^T |g(s, y, z, \theta, \phi)|^2ds]< +
\infty,$ for any $(y, z, \theta, \phi)$.

\section{Existence and uniqueness theorem}

In this section, we will mainly study the existence and uniqueness
of the solution to anticipated BDSDE (\ref{equation:main}). For this
purpose, we first consider a simple case when the coefficients $f$
and $g$ do not depend on the value or the future value of $(Y, Z)$:
\begin{equation}\label{eq only t}
Y_t = \xi_T+ \int_t^T f(s)ds + \int_t^T g(s) d\overleftarrow{B}_s -
\int_t^T Z_s d{W}_s,\ \ t\in[0, T],
\end{equation}
where $f\in L_{\mathcal{G}}^2(0, T; \mathbb{R}^m)$, $g\in
L_{\mathcal{G}}^2(0, T; \mathbb{R}^{m\times l})$ and $\xi_T \in
L^2(\mathcal{G}_T; \mathbb{R}^m)$.

\begin{theorem}\label{thm only t}
Given $\xi_T \in L^2(\mathcal{G}_T; \mathbb{R}^m)$, BDSDE (\ref{eq
only t}) has a unique solution $(Y, Z) \in L_\mathcal{G}^2(0, T;
\mathbb{R}^m)\times L_\mathcal{G}^2(0, T; \mathbb{R}^{m\times d})$.
\end{theorem}

\begin{proof}
To prove the existence, we define a filtration by
$$\mathcal{H}_t := \mathcal{F}_{0, t}^W \vee \mathcal{F}_{0, T+K}^B,\ t\in [0, T +
K]$$ and a $\mathcal{H}_t$-square integrable martingale
$$M_t := E^{\mathcal{H}_t}[\xi_T + \int_0^T f(s)ds +\int_0^T g(s) d\overleftarrow{B}_s],\ \ t\in [0, T].$$
Thanks to It\^{o}'s martingale representation theorem, there exists
a process $Z \in L_{\mathcal{H}}^2 (0, T; \mathbb{R}^{m \times d})$
such that
$$M_t = M_0 + \int_0^t Z_s dW_s, \ \ t\in [0, T],$$
which implies $$M_t = M_T - \int_t^T Z_s dW_s, \ \ t\in [0, T].$$
Hence
$$E^{\mathcal{H}_t}[\xi_T + \int_0^T f(s)ds +\int_0^T g(s) d\overleftarrow{B}_s]=\xi_T + \int_0^T f(s)ds +\int_0^T g(s) d\overleftarrow{B}_s- \int_t^T Z_s dW_s.$$
Subtract $\int_0^t f(s)ds +\int_0^t g(s) d\overleftarrow{B}_s$ from
both sides, then we have
$$Y_t=\xi_T + \int_t^T f(s)ds +\int_t^T g(s) d\overleftarrow{B}_s- \int_t^T Z_s dW_s,$$
where
$$Y_t := E^{\mathcal{H}_t}[\xi_T + \int_t^T f(s)ds +\int_t^T g(s) d\overleftarrow{B}_s].$$
Next we show that $(Y, Z)$ are in fact $\mathcal{G}_t$-adapted. In
fact, it is obvious that $$Y_t = E[\Theta|\mathcal{G}_t \vee
\mathcal{F}_{0, t}^B],$$ where $\Theta:=\xi_T + \int_t^T f(s)ds
+\int_t^T g(s) d\overleftarrow{B}_s$ is $\mathcal{F}_{0, T}^W \vee
\mathcal{F}_{t, T+K}^B$ measurable. Note that $\mathcal{F}_{0, t}^B$
is independent of $\mathcal{G}_t \vee \sigma(\Theta)$, then we know
$$Y_t = E^{\mathcal{G}_t} [\Theta].$$
Now $$\int_t^T Z_s dW_s=\xi_T + \int_t^T f(s)ds +\int_t^T g(s)
d\overleftarrow{B}_s- Y_t,$$ and the right side is $\mathcal{F}_{0,
T}^W \vee \mathcal{F}_{t, T+K}^B$ measurable. Then from It\^{o}'s
martingale representation theorem, $(Z_s)_{s\in[t, T]}$ is
$\mathcal{F}_{0, s}^W \vee \mathcal{F}_{t, T+K}^B$ adapted, which
implies $Z_s$ is $\mathcal{F}_{0, s}^W \vee \mathcal{F}_{t, T+K}^B$
measurable for any $t \leq s$. Thus $Z_s$ is $\mathcal{F}_{0, s}^W
\vee \mathcal{F}_{s, T+K}^B$ measurable.

To show the uniqueness. We suppose that $(\bar{Y}, \bar{Z})$ is the
difference of two solutions. Then
$$\bar{Y}_t + \int_t^T \bar{Z}_sdW_s=0, \ \ t\in [0, T].$$
Hence $$E |\bar{Y}_t|^2 + E \int_t^T |\bar{Z}_s|^2 ds =0,$$ which
implies $\bar{Y}_t \equiv 0,\ a.s.$ and $\bar{Z}_t \equiv 0\ a.s.,\
a.e..$
\end{proof}

Now we establish the main result of this part.

\begin{theorem}\label{thm existence uniqueness}
Assume that $(a1)$-$(a2)$ and $(H1)$-$(H4)$ hold. Then for given
$(\xi, \eta) \in S_{\mathcal{G}}^2(T, T+K; \mathbb{R}^m)\times
L_\mathcal{G}^2(T, T+K; \mathbb{R}^{m \times d})$, the anticipated
BDSDE (\ref{equation:main}) has a unique solution $(Y, Z) \in
S_\mathcal{G}^2(0, T+K; \mathbb{R}^m)\times L_\mathcal{G}^2(0, T+K;
\mathbb{R}^{m\times d})$.
\end{theorem}

\begin{proof}
Denote by $\mathcal{S}$ the space of $(Y, Z)\in L_{\mathcal{G}}^2(0,
T+K; \mathbb{R}^m)\times L_{\mathcal{G}}^2(0, T+K; \mathbb{R}^{m
\times d})$ such that $(Y_t, Z_t)_{t\in [T, T+K]}= (\xi_t,
\eta_t)_{t\in [T, T+K]}$. Given $(y, z)\in \mathcal{S}$, we consider
the following equation:
\begin{equation}\label{equation: 3}
\left\{
\begin{tabular}{rlll}
$-dY_t$ &=& $f(t, y_t, z_t, y_{t+\delta(t)}, z_{t+\zeta(t)})dt$ &
\vspace{2mm}
\\ && $+ g(t, y_t, z_t, y_{t+\delta(t)}, z_{t+\zeta(t)}) d\overleftarrow{B}_t - Z_t d{W}_t, $
& $ t\in[0, T];$ \vspace{2mm}
\\
$Y_t$ &=& $\xi_t, $ & $t\in[T, T+K];$ \vspace{2mm}
\\
$Z_t$ &=& $\eta_t, $ & $t\in[T, T+K].$
\end{tabular}\right.
\end{equation}
It is obvious that the above equation is equivalent to the BDSDE
\begin{equation*}
\left\{
\begin{tabular}{rlll}
$-d\tilde{Y}_t$ &=& $f(t, y_t, z_t, y_{t+\delta(t)},
z_{t+\zeta(t)})dt$ & \vspace{2mm}
\\ && $+ g(t, y_t, z_t, y_{t+\delta(t)}, z_{t+\zeta(t)}) d\overleftarrow{B}_t - \tilde{Z}_t d{W}_t, $
& $ t\in[0, T];$ \vspace{2mm}
\\
$\tilde{Y}_T$ &=& $\xi_T\in \mathcal{G}_T,$
\end{tabular}\right.
\end{equation*}
which admits a unique solution in the space $S_{\mathcal{G}}^2(0, T;
\mathbb{R}^m)\times L_{\mathcal{G}}^2(0, T; \mathbb{R}^{m \times
d})$ according to Theorem \ref{thm only t}. Thus BDSDE
(\ref{equation: 3}) has a unique solution in $\mathcal{S}$. Define a
mapping $I$ from $\mathcal{S}$ into itself by $(Y, Z)=I(y, z)$, then
$(Y, Z)$ is the unique solution of BDSDE (\ref{equation: 3}).

Let $(y^\prime, z^\prime)$ be another element of $\mathcal{S}$, and
$(Y^\prime, Z^\prime)=I(y^\prime, z^\prime)$. We make the following
notations:
\begin{align*}
& \bar{y}=  y-y^\prime,\ \bar{z}=z-z^\prime,\ \bar{Y}=Y-Y^\prime,\
\bar{Z}=Z-Z^\prime,  \\&  \bar{f}_t =  f(t, y_t, z_t,
y_{t+\delta(t)}, z_{t+\zeta(t)})-f(t, y_t^\prime, z_t^\prime,
y_{t+\delta(t)}^\prime, z_{t+\zeta(t)}^\prime), \\
& \bar{g}_t =  g(t, y_t, z_t, y_{t+\delta(t)}, z_{t+\zeta(t)})-g(t,
y_t^\prime, z_t^\prime, y_{t+\delta(t)}^\prime,
z_{t+\zeta(t)}^\prime).
\end{align*}
For any $\beta >0$, apply It\^{o}'s formula to $e^{\beta
t}|\bar{Y}_t|^2$,
\begin{equation*}
\begin{tabular}{rll}
& & $e^{\beta t} |\bar{Y}_t|^2 + \int_t^T e^{\beta s}
[\beta|\bar{Y}_s|^2+|\bar{Z}_s|^2]ds$ \vspace{2mm} \\ && =  $2
\int_t^{T} e^{\beta s} \bar{Y}_s \bar{f}_sds + \int_t^T e^{\beta s}
|\bar{g}_s|^2 ds + 2\int_t^T e^{\beta s} \bar{Y}_s \bar{g}_s
d\overleftarrow{B}_s - 2\int_t^T e^{\beta s} \bar{Y}_s \bar{Z}_s
d{W}_s.$
\end{tabular}
\end{equation*}
Take mathematical expectation on both sides, then we have
\begin{equation*}
e^{\beta t} E|\bar{Y}_t|^2 + E \int_t^T e^{\beta s}
[\beta|\bar{Y}_s|^2+|\bar{Z}_s|^2]ds=2 E \int_t^{T} e^{\beta s}
\bar{Y}_s \bar{f}_s ds + E \int_t^T e^{\beta s}|\bar{g}_s|^2 ds.
\end{equation*}
Hence from $(A1)$, $(A2)$ and the inequality $2ab \leq \lambda a^2 +
\frac{1}{\lambda} b^2$,
\begin{equation*}
\begin{tabular}{rll}
& & $e^{\beta t} E|\bar{Y}_t|^2 + E \int_t^T e^{\beta s}
[\beta|\bar{Y}_s|^2+|\bar{Z}_s|^2]ds$ \vspace{2mm}
\\& & $ \leq
 E \int_t^{T} e^{\beta s} [\lambda |\bar{Y}_s|^2+ \frac{1}{\lambda}
|\bar{f}_s|^2] ds + E \int_t^T e^{\beta s}|\bar{g}_s|^2 ds$
\vspace{2mm}
\\
&& $\leq E\int_t^T e^{\beta s} [\lambda |\bar{Y}_s|^2+
(\frac{c}{\lambda}+c)(|\bar{y}_s|^2 + |\bar{y}_{s+\delta(s)}|^2)+
(\frac{c}{\lambda}+\alpha_1) |\bar{z}_s|^2 +
(\frac{c}{\lambda}+\alpha_2) |\bar{z}_{s+\zeta(s)}|^2 ] ds$
\vspace{2mm}
\\
&& $\leq E\int_t^{T+K} e^{\beta s} [\lambda |\bar{Y}_s|^2+
(\frac{c}{\lambda}+c)(1 + M)|\bar{y}_s|^2 +
(\frac{c}{\lambda}(1+M)+\alpha_1 +\alpha_2 M) |\bar{z}_s|^2] ds,$
\end{tabular}
\end{equation*}
which implies
\begin{equation*}
\begin{tabular}{rll}
& & $E \int_t^{T+K} e^{\beta s}
[(\beta-\lambda)|\bar{Y}_s|^2+|\bar{Z}_s|^2]ds$ \vspace{2mm}
\\
&& $\leq E\int_t^{T+K} e^{\beta s} [(\frac{c}{\lambda}+c)(1 +
M)|\bar{y}_s|^2 + (\frac{c}{\lambda}(1+M)+\alpha_1 +\alpha_2 M)
|\bar{z}_s|^2] ds$ \vspace{2mm}
\\
&& $= (\frac{c}{\lambda}(1+M)+\alpha_1 +\alpha_2 M) E\int_t^{T+K}
e^{\beta s} [\frac{c(1+ \lambda)(1+M)}{c(1+ M) + \lambda
(\alpha_1+\alpha_2 M)}|\bar{y}_s|^2 + |\bar{z}_s|^2] ds.$
\end{tabular}
\end{equation*}

Hence if we choose $\lambda=\lambda_0$ satisfying
$\bar{c}:=\frac{c}{\lambda_0}(1+M)+\alpha_1 +\alpha_2 M <1 $, choose
$\beta=\lambda_0 + \frac{c(1+ \lambda_0)(1+M)}{c(1+ M) + \lambda_0
(\alpha_1+\alpha_2 M)}$, and denote $\gamma:=\frac{c(1+
\lambda_0)(1+M)}{c(1+ M) + \lambda_0 (\alpha_1+\alpha_2 M)}$, then
we deduce
\begin{equation*}
E \int_t^{T+K} e^{\beta s} [\gamma|\bar{Y}_s|^2+|\bar{Z}_s|^2]ds
\leq \bar{c} E\int_t^{T+K} e^{\beta s} [\gamma|\bar{y}_s|^2 +
|\bar{z}_s|^2] ds.
\end{equation*}
Thus $I$ is a strict contraction on $\mathcal{S}$ and it has a
unique fixed point $(Y, Z)\in \mathcal{S}$. Now due to
Burkholder-Davis-Gundy inequality, it is easy to check that $Y \in
S_{\mathcal{G}}^2(0, T+K; \mathbb{R}^m)$. The proof is complete.
\end{proof}

\begin{remark}
In the proof of Theorem \ref{thm existence uniqueness}, we use the
norm
\begin{equation*}
|(Y, Z)|_{(\beta, \gamma)} \equiv \{E \int_0^{T+K}
e^{\beta_s}(\gamma |Y_s|^2+|Z_s|^2)ds\}^{\frac{1}{2}},
\end{equation*}
which is very convenient for us to establish a strict contraction
mapping. In fact, it is obvious that this new norm is equivalent to
both norms $|(Y, Z)|_{(\beta, 1)}$ and $|(Y, Z)|_{(0, 1)}$, and the
latter is just the general norm defined on the space
$L_\mathcal{G}^2(0, T+K; \mathbb{R}^m)\times L_\mathcal{G}^2(0, T+K;
\mathbb{R}^{m\times d})$.

\end{remark}

\section{Comparison theorem}

In this part, we are concerned with the following $1$-dimensional
anticipated BDSDEs:
\begin{equation}\label{equation:comparison j}
\left\{
\begin{tabular}{rlll}
$-dY_t^j$ &=& $f^j(t, Y_t^j, Z_t^j, Y_{t+\delta(t)}^j,
Z_{t+\zeta(t)}^j)dt+ g(t, Y_t^j, Z_t^j) d\overleftarrow{B}_t - Z_t^j
d{W}_t, $ & $ t\in[0, T];$ \vspace{2mm}
\\
$Y_t^j$ &=& $\xi_t^j, $ & $t\in[T, T+K],$ \vspace{2mm}
\\ $Z_t^j$ &=& $\eta_t^j, $ & $t\in[T, T+K],$
\end{tabular}\right.
\end{equation}
where $j=1, 2$, and $(a1)$-$(a2)$, $(H1)$-$(H4)$ hold. Then by
Theorem \ref{thm existence uniqueness}, (\ref{equation:comparison
j}) has a unique solution.

Our objective is to obtain a comparison result for these two
equations. For this purpose, we first consider a simple case when
the coefficients $f^j$ and $g$ do not depend on the future value of
$(Y^j, Z^j)$:
\begin{equation}\label{eq j only t}
Y_t^j = \xi_T^j+ \int_t^T f^j(s, Y_s^j, Z_s^j)ds + \int_t^T g(s,
Y_s^j, Z_s^j) d\overleftarrow{B}_s - \int_t^T Z_s^j d{W}_s,\ \
t\in[0, T].
\end{equation}

\begin{theorem}\label{thm comparison only t}
Let $(Y^{j}, Z^{j})\in S_\mathcal{G}^2(0, T; \mathbb{R})\times
L_\mathcal{G}^2(0, T; \mathbb{R}^d)$ $(j=1, 2)$ be the unique
solutions to BDSDEs (\ref{eq j only t}) respectively. If $\xi_T^1
\geq \xi_T^2,\ a.s.,$ and for any $(t, y, z) \in [0,
T]\times\mathbb{R}\times \mathbb{R}^d$, $f^1(t, y, z) \geq f^2(t, y,
z),\ a.s.$, then $Y_t^1 \geq Y_t^2,\ a.s.,$ for all $t\in [0,T]$.
\end{theorem}
\begin{proof}
Denote $$\bar{Y}_t:=Y_t^2-Y_t^1,\ \ \bar{Z}_t:=Z_t^2-Z_t^1,\ \
\bar{\xi}_T:=\xi_T^2-\xi_T^1,$$ then $(\bar{Y}, \bar{Z})$ satisfies
$$\bar{Y}_t=\bar{\xi}_T + \int_t^T [f^2(s, Y_s^2, Z_s^2)-f^1(s, Y_s^1, Z_s^1)]ds + \int_t^T [g(s, Y_s^2, Z_s^2)-g(s, Y_s^1, Z_s^1)]d \overleftarrow{B}_s - \int_t^T \bar{Z}_sdW_s.$$
Applying It\^{o}'s formula to $|\bar{Y}_t^+|^2$, we have
\begin{equation*}
\begin{tabular}{rll}
$|\bar{Y}_t^+|^2$ &
= & $|\bar{\xi}_T^+|^2 + 2\int_t^T \bar{Y}_s^+ [f^2(s, Y_s^2, Z_s^2)-f^1(s, Y_s^1, Z_s^1)]ds $ \vspace{5mm} \\
& & $+ 2\int_t^T \bar{Y}_s^+ [g(s, Y_s^2, Z_s^2)-g(s, Y_s^1, Z_s^1)]d \overleftarrow{B}_s- 2\int_t^T \bar{Y}_s^+ \bar{Z}_sdW_s$\vspace{5mm} \\
&& $-\int_t^T 1_{\{Y_s^2 > Y_s^1\}}|\bar{Z}_s|^2ds +\int_t^T 1_{\{Y_s^2 > Y_s^1\}} |g(s, Y_s^2, Z_s^2)-g(s, Y_s^1, Z_s^1)|^2ds.$\\
\end{tabular}
\end{equation*}
Taking expectation on both sides and noting that $\xi_T^1 \geq
\xi_T^2$, we get
\begin{equation*}
\begin{tabular}{rll}
$E|\bar{Y}_t^+|^2 + E-\int_t^T 1_{\{Y_s^2 > Y_s^1\}}|\bar{Z}_s|^2ds$
& = & $2E\int_t^T \bar{Y}_s^+ [f^2(s, Y_s^2, Z_s^2)-f^1(s, Y_s^1, Z_s^1)]ds $ \vspace{5mm} \\
& & $+ E\int_t^T 1_{\{Y_s^2 > Y_s^1\}} |g(s, Y_s^2, Z_s^2)-g(s, Y_s^1, Z_s^1)|^2ds.$\\
\end{tabular}
\end{equation*}
While,
\begin{equation*}
\begin{tabular}{rll}
& $2E\int_t^T \bar{Y}_s^+ [f^2(s, Y_s^2, Z_s^2)-f^1(s, Y_s^1,
Z_s^1)]ds$ \vspace{5mm}\\
 &=$2 E\int_t^T \bar{Y}_s^+
[f^2(s, Y_s^2, Z_s^2)-f^1(s, Y_s^2, Z_s^2)+f^1(s, Y_s^2,
Z_s^2)-f^1(s, Y_s^1,
Z_s^1)]ds$\vspace{5mm}\\
& $\leq 2 E\int_t^T \bar{Y}_s^+ |f^1(s, Y_s^2, Z_s^2)-f^1(s,
Y_s^1, Z_s^1)|ds \leq 2\sqrt{c}E\int_t^T \bar{Y}_s^+ [|\bar{Y}_s|+|\bar{Z}_s|]ds$ \vspace{5mm}\\
& $\leq (2\sqrt{c}+\frac{c}{1-\alpha_1})E\int_t^T
|\bar{Y}_s^+|^2ds+(1-\alpha_1)\int_t^T 1_{\{Y_s^2 >
Y_s^1\}}|\bar{Z}_s|^2ds,$
\end{tabular}
\end{equation*}
and
\begin{equation*}
\begin{tabular}{rll}
& $E\int_t^T 1_{\{Y_s^2 > Y_s^1\}} |g(s, Y_s^2, Z_s^2)-g(s, Y_s^1, Z_s^1)|^2ds$ \vspace{5mm}\\
& $\leq E\int_t^T 1_{\{Y_s^2 >
Y_s^1\}}[c|\bar{Y}_s|^2+\alpha_1|\bar{Z}_s|^2]ds$ \vspace{5mm}\\
& $\leq c E\int_t^T |\bar{Y}_s^+|^2ds +\alpha_1 \int_t^T 1_{\{Y_s^2
> Y_s^1\}}|\bar{Z}_s|^2ds.$
\end{tabular}
\end{equation*}
Then, thanks to the above inequalities, we obtain $$E|\bar{Y}_t^+|^2
\leq (c+2\sqrt{c}+\frac{c}{1-\alpha_1})E\int_t^T
|\bar{Y}_s^+|^2ds,$$ which implies
$$E|\bar{Y}_t^+|^2=0,\ \ {\rm{for\ all}}\ t\in [0, T].$$
Therefore $Y_t^1 \geq Y_t^2,\ a.s.,$ for all $t\in [0, T]$.
\end{proof}

From now on, we consider the anticipated BDSDEs
(\ref{equation:comparison j}). We give the following result. For the
proof, the reader is referred to \cite{Xu}.

\begin{proposition}
Putting $t_0=T$, we define by iteration
\begin{equation*}
t_i:=\min\{t \in [0, T]: \min\{s+ \delta(s),\ s+\zeta(s)\}\geq
t_{i-1}, {\rm{\ for\ all}}\ s\in [t, T]\},\ \ \ i\geq 1.
\end{equation*}
Set $N:=\max\{i: t_{i-1}>0\}$. Then $N$ is finite, $t_N=0$ and
\begin{equation*}
[0, T]=[0, t_{N-1}]\cup [t_{N-1}, t_{N-2}]\cup \cdots \cup [t_2,
t_1]\cup [t_1, T].
\end{equation*}
\end{proposition}

\begin{proposition}\label{prop2}
For $j=1, 2$, suppose that $(Y^{j}, Z^{j})$ is the unique solution
to the anticipated BDSDE (\ref{equation:comparison j}). Then for
fixed $i\in {\{1, 2,\ldots, N\}},$ over time interval $[t_i,
t_{i-1}]$, (\ref{equation:comparison j}) is equivalent to
\begin{equation}\label{equation:in prop 2}
\left\{
\begin{tabular}{rlll}
$-d\bar{Y}_t^j$ &=& $f^j(t, \bar{Y}_t^j, \bar{Z}_t^j,
\bar{Y}_{t+\delta(t)}^j, \bar{Z}_{t+\zeta(t)}^j)dt+ g(t,
\bar{Y}_t^j, \bar{Z}_t^j) d\overleftarrow{B}_t - \bar{Z}_t^j d{W}_t,
$ & $ t\in[t_i, t_{i-1}];$ \vspace{2mm}
\\
$\bar{Y}_t^j$ &=& $Y_t^j, $ & $t\in[t_{i-1}, T+K],$  \vspace{2mm}
\\
$\bar{Z}_t^j$ &=& $Z_t^j, $ & $t\in[t_{i-1}, T+K],$
\end{tabular}\right.
\end{equation}
which is also equivalent to the following BDSDE with terminal
condition $Y_{t_{i-1}}^{j}$:
\begin{equation}\label{equation:in prop 3}
\tilde{Y}_t^{j}=Y_{t_{i-1}}^{j}+\int_t^{t_{i-1}} f^j(s,
\tilde{Y}_s^j, \tilde{Z}_s^j, Y_{s+\delta(s)}^j,
Z_{s+\zeta(s)}^j)ds+ \int_t^{t_{i-1}} g(s, \tilde{Y}_s^j,
\tilde{Z}_s^j) d\overleftarrow{B}_s - \int_t^{t_{i-1}} \tilde{Z}_s^j
d{W}_s.
\end{equation}
That is to say,
$$
Y_t^{j}=\bar{Y}_t^{j}=\tilde{Y}_t^{j},\
Z_t^{j}=\bar{Z}_t^{j}=\tilde{Z}_t^{j}=\frac{d\langle \tilde{Y}^{j},
W \rangle_t}{d t},\ t\in [t_i, {t_{i-1}}],\ j=1,2.,$$ where $\langle
\tilde{Y}^{j}, W \rangle$ is the variation process generated by
$\tilde{Y}^{j}$ and the Brownian motion $W$.
\end{proposition}

The main result of this part is

\begin{theorem}\label{thm comparison}
Let $(Y^{j}, Z^{j})\in S_\mathcal{G}^2(0, T+K; \mathbb{R})\times
L_\mathcal{G}^2(0, T+K; \mathbb{R}^d)$ $(j=1, 2)$ be the unique
solutions to anticipated BDSDEs (\ref{equation:comparison j})
respectively. If
\item{(i)} $\xi_s^{1}\geq \xi_s^{2}, s\in [T, T+K], a.e., a.s.;$
\item{(ii)} for all $t\in [0, T]$, $(y, z)\in
\mathbb{R}\times \mathbb{R}^d,$ $\theta^{j}\in S_\mathcal{G}^2(t,
T+K; \mathbb{R})$ $(j=1, 2)$ such that $\theta^{1} \geq \theta^{2},$
$\{\theta_{r}^{j}\}_{r\in[t, T]}$ is a continuous semimartingale and
$(\theta_{r}^{j})_{r\in [T, T+K]}=(\xi_r^{j})_{r\in [T, T+K]}$,
\begin{equation}\label{equation:condition 1}
f^1(t, y, z, \theta_{t+\delta(t)}^{1}, \eta_{t+\zeta(t)}^{1}) \geq
f^2(t, y, z, \theta_{t+\delta(t)}^{2}, \eta_{t+\zeta(t)}^{2}),\ \
a.e., a.s.,
\end{equation}
\begin{equation}\label{equation:condition 2}
f^1(t, y, z, \theta_{t+\delta(t)}^{1}, \frac{d\langle \theta^{1},
W\rangle_r}{d r}|_{r=t+\zeta(t)}) \geq f^2(t, y, z,
\theta_{t+\delta(t)}^{2}, \frac{d\langle \theta^{2}, W\rangle_r}{d
r}|_{r=t+\zeta(t)}),\ \ a.e., a.s.,
\end{equation}
\begin{equation}\label{equation:condition 3}
f^1(t, y, z, \xi_{t+\delta(t)}^{1}, \frac{d\langle \theta^{1},
W\rangle_r}{d r}|_{r=t+\zeta(t)}) \geq f^2(t, y, z,
\xi_{t+\delta(t)}^{2}, \frac{d\langle \theta^{2}, W\rangle_r}{d
r}|_{r=t+\zeta(t)}),\ \ a.e., a.s.,
\end{equation}
then $Y_t^{1} \geq Y_t^{2}, a.e., a.s..$
\end{theorem}
\begin{proof}
Consider the anticipated BDSDE (\ref{equation:comparison j}) one
time interval by one time interval. For the first step, we consider
the case when $t\in [t_1, T]$. According to Proposition \ref{prop2},
we can equivalently consider
\begin{equation*}
\tilde{Y}_t^{j}=\xi_T^{j}+\int_t^T f^j(s, \tilde{Y}_s^j,
\tilde{Z}_s^j, \xi_{s+\delta(s)}^j, \eta_{s+\zeta(s)}^j)ds+ \int_t^T
g(s, \tilde{Y}_s^j, \tilde{Z}_s^j) d\overleftarrow{B}_s - \int_t^T
\tilde{Z}_s^j d{W}_s,
\end{equation*}
from which we have
\begin{equation}\label{equation:Z t1 T}
{Z}_t^{j}=\tilde{Z}_t^{j}=\frac{d\langle \tilde{Y}^{j},
W\rangle_t}{d t},\ \ t\in [t_1, T].
\end{equation}
Noticing that $\xi^{j}\in S_\mathcal{G}^2(T, T+K; \mathbb{R})$
$(j=1, 2)$ and $\xi^{1}\geq \xi^{2},$ from (\ref{equation:condition
1}) in (ii), we can get, for $s\in [t_1, T],$ $y\in \mathbb{R}$,
$z\in \mathbb{R}^d,$
\begin{center}
$f^1(s, y, z, \xi_{s+\delta(s)}^{1}, \eta_{s+\zeta(s)}^{1}) \geq
f^2(s, y, z, \xi_{s+\delta(s)}^{2}, \eta_{s+\zeta(s)}^{2}). $
\end{center}
According to Theorem \ref{thm comparison only t}, we can get
\begin{center}
$ \tilde{Y}_t^{1}\geq \tilde{Y}_t^{2},\ \ t\in[t_1, T],\ \
a.e.,a.s., $
\end{center}
which implies
\begin{equation}\label{equation:Y t1 T}
Y_t^{(1)}\geq Y_t^{(2)},\ \ t\in[t_1, T+K],\ \ a.e.,a.s..
\end{equation}

For the second step, we consider the case when $t\in [t_2, t_1]$.
Similarly, according to Proposition \ref{prop2}, we can consider the
following BSDE equivalently:
\begin{equation*}
\tilde{\tilde{Y}}_t^{j}=Y_{t_1}^{j}+\int_t^{t_1} f^j(s,
\tilde{\tilde{Y}}_s^j, \tilde{\tilde{Z}}_s^j, Y_{s+\delta(s)}^j,
Z_{s+\zeta(s)}^j)ds+ \int_t^{t_1} g(s, \tilde{\tilde{Y}}_s^j,
\tilde{\tilde{Z}}_s^j) d\overleftarrow{B}_s - \int_t^{t_1}
\tilde{\tilde{Z}}_s^j d{W}_s,
\end{equation*}
from which we have ${Z}_t^{j}=\tilde{\tilde{Z}}_t^{j}=\frac{d\langle
\tilde{\tilde{Y}}^{j}, W \rangle_t}{d t}$ for $t\in [t_2, t_1]$.
Noticing (\ref{equation:Z t1 T}) and (\ref{equation:Y t1 T}),
according to (ii), we have, for $s\in [t_2, t_1]$, $y\in
\mathbb{R}$, $z\in \mathbb{R}^d,$
\begin{center}
$f^1(s, y, z, Y_{s+\delta(s)}^{1}, Z_{s+\zeta(s)}^{1}) \geq f^2(s,
y, z, Y_{s+\delta(s)}^{2}, Z_{s+\zeta(s)}^{2}). $
\end{center}
Applying Theorem \ref{thm comparison only t} again, we can finally
get
\begin{center}
$ Y_t^{1}\geq Y_t^{2},\ \ t\in[t_2, t_1],\ \ a.e.,a.s.. $
\end{center}

Similarly to the above steps, we can give the proofs for the other
cases when $t\in [t_3, t_2],$ $[t_4, t_3],$ $\cdots,$ $[t_N,
t_{N-1}].$
\end{proof}

\begin{example}
Now suppose that we are facing with the following two ABDSDEs:
\begin{equation*}
\left\{
\begin{tabular}{rlll}
$-dY_t^1$ &=& $E^{\mathcal{F}_t}[Y_{t+\delta(t)}^{1}+\sin
(2Y_{t+\delta(t)}^{1})+|Z_{t+\zeta(t)}^{1}|+2]dt$ \vspace{2mm}
\\
&& $ + [Y_{t}^{1}+\frac{1}{\sqrt{3}}|Z_{t}^{1}|]
d\overleftarrow{B}_t - Z_t^1 d{W}_t, $ & $ t\in[0, T];$ \vspace{2mm}
\\
$Y_t^1$ &=& $\xi_t^1, $ & $t\in[T, T+K],$ \vspace{2mm}
\\
$Z_t^1$ &=& $\eta_t^1, $ & $t\in[T, T+K],$
\end{tabular}\right.
\end{equation*}
\begin{equation*}
\left\{
\begin{tabular}{rlll}
$-dY_t^2$ &=& $E^{\mathcal{F}_t}[Y_{t+\delta(t)}^{2}+2 |\cos
Y_{t+\delta(t)}^{2}|+\sin Z_{t+\zeta(t)}^{2}-2]dt$ \vspace{2mm}
\\
&& $ + [Y_{t}^{2}+\frac{1}{\sqrt{3}}|Z_{t}^{2}|]
d\overleftarrow{B}_t - Z_t^2 d{W}_t, $ & $ t\in[0, T];$ \vspace{2mm}
\\
$Y_t^2$ &=& $\xi_t^2, $ & $t\in[T, T+K],$ \vspace{2mm}
\\
$Z_t^2$ &=& $\eta_t^2, $ & $t\in[T, T+K],$
\end{tabular}\right.
\end{equation*}
where $\xi_t^{(1)}\geq \xi_t^{(2)}, t\in [T, T+K].$

It is obvious that
\begin{center}
$x+\sin (2x)+|u|+2 \geq y+2 |\cos y| + \sin v-2,\ for\ all\ x\geq y,
x, y\in \mathbb{R}, u, v \in \mathbb{R}^d,$
\end{center}
which implies (\ref{equation:condition 1})-(\ref{equation:condition
3}), then according to Theorem \ref{thm comparison}, we get
$Y_t^{1}\geq Y_t^{2},\ a.e.,\ a.s..$
\end{example}

\begin{remark}
By the same way, for the case when $\delta=\zeta$,
(\ref{equation:condition 1})-(\ref{equation:condition 3}) can be
replaced by (\ref{equation:condition 2}) together with
\begin{center}
$ f^1(t, y, z, \xi_{t+\delta(t)}^{1}, \eta_{t+\zeta(t)}^{1}) \geq
f^2(t, y, z, \xi_{t+\delta(t)}^{2}, \eta_{t+\zeta(t)}^{2}),\ \ a.e.,
a.s.. $
\end{center}
\end{remark}

For a special case when $f^1$ and $f^2$ are independent of the
anticipated term $Z$, we easily get the following comparison result.

\begin{theorem}\label{thm comparison no Z}
Let $(Y^{j}, Z^{j})\in S_\mathcal{G}^2(0, T+K; \mathbb{R})\times
L_\mathcal{G}^2(0, T+K; \mathbb{R}^d)$ $(j=1, 2)$ be the unique
solutions to the following ABDSDEs respectively:
\begin{equation*}
\left\{
\begin{tabular}{rlll}
$-dY_t^j$ &=& $f^j(t, Y_t^j, Z_t^j, Y_{t+\delta(t)}^j)dt+ g(t,
Y_t^j, Z_t^j) d\overleftarrow{B}_t - Z_t^j d{W}_t, $ & $ t\in[0,
T];$ \vspace{2mm}
\\
$Y_t^j$ &=& $\xi_t^j, $ & $t\in[T, T+K].$
\end{tabular}\right.
\end{equation*}
If
\item{(i)} $\xi_s^{1}\geq \xi_s^{2}, s\in [T, T+K], a.e., a.s.;$
\item{(ii)} for all $t\in [0, T]$, $(y, z)\in
\mathbb{R}\times \mathbb{R}^d,$ $\theta^{j}\in S_\mathcal{G}^2(t,
T+K; \mathbb{R})$ $(j=1, 2)$ such that $\theta^{1} \geq \theta^{2}$
and $(\theta_{r}^{j})_{r\in [T, T+K]}=(\xi_r^{j})_{r\in [T, T+K]}$,
\begin{equation}\label{equation:no z}
f^1(t, y, z, \theta_{t+\delta(t)}^{1}) \geq f^2(t, y, z,
\theta_{t+\delta(t)}^{2}),\ \ a.e., a.s.,
\end{equation}
then $Y_t^{1} \geq Y_t^{2}, a.e., a.s..$
\end{theorem}

\begin{remark}
The coefficients $f^1$ and $f^2$ will satisfy (\ref{equation:no z}),
if for any $(t, y, z)\in [0, T]\times \mathbb{R}\times
\mathbb{R}^d,$ $\theta\in L_\mathcal{G}^2(t, T+K; \mathbb{R}),$
$r\in [t, T+K],$ $f^1(t, y, z, \theta_r)\geq f^2(t, y, z,
\theta_r),$ together with one of the following:
\item{(i)} for
any $(t, y, z)\in [0, T]\times \mathbb{R}\times \mathbb{R}^d,$
$f^1(t, y, z, \cdot)$ is increasing, i.e., $f^1(t, y, z,
\theta_r)\geq f^1(t, y, z, \theta_r^\prime),$ if $\theta \geq
\theta^\prime,$ $\theta, \theta^\prime\in L_\mathcal{G}^2(t, T+K;
\mathbb{R}),$ $r\in [t, T+K];$
\item{(ii)} for any $(t, y, z)\in [0, T]\times \mathbb{R}\times \mathbb{R}^d,$
$f^2(t, y, z, \cdot)$ is increasing, i.e., $f^2(t, y, z,
\theta_r)\geq f^2(t, y, z, \theta_r^\prime),$ if $\theta \geq
\theta^\prime,$ $\theta, \theta^\prime\in L_\mathcal{G}^2(t, T+K;
\mathbb{R}),$ $r\in [t, T+K].$
\end{remark}

\begin{remark}\label{remark}
The coefficients $f^1$ and $f^2$ will satisfy (\ref{equation:no z}),
if
\begin{center}
$ f^1(t, y, z, \theta_r)\geq \tilde{f}(t, y, z, \theta_r)\geq f^2(t,
y, z, \theta_r), $
\end{center}
for any $(t, y, z)\in [0, T]\times \mathbb{R}\times \mathbb{R}^d,$
$\theta \in L_\mathcal{G}^2(t, T+K; \mathbb{R}), r\in [t, T+K].$
Here the function $\tilde{f}(t, y, z, \cdot)$ is increasing, for any
$(t, y, z)\in [0, T]\times \mathbb{R}\times \mathbb{R}^d,$
  i.e., $\tilde{f}(t, y, z, \theta_r)\geq \tilde{f}(t, y, z,
\theta_r^\prime),$ if $\theta_r\geq \theta_r^\prime,$ $\theta,
\theta^\prime\in L_\mathcal{G}^2(t, T+K; \mathbb{R}), r\in [t,
T+K].$
\end{remark}

\begin{example} The following three functions satisfy
the conditions in Remark \ref{remark}: $f^1(t, y, z,
\theta_r)=E^{\mathcal{F}_t}[\theta_r- \sin (2\theta_r) + 2]$,
$\tilde{f}(t, y, z, \theta_r)=E^{\mathcal{F}_t}[\theta_r+\cos
\theta_r]$, $f^2(t, y, z, \theta_r)=E^{\mathcal{F}_t}[\theta_r+2\cos
\theta_r-1].$
\end{example}

\section{A duality result between delayed DSDEs and anticipated BDSDEs}

In this part we will establish a duality between the following
anticipated BDSDE
\begin{equation}\label{eq duality bdsde}
\left\{
\begin{tabular}{rlll}
$-dY_t$ &=& $([\mu_t+\kappa_t^2] Y_t+ \bar{\mu}_t E^{\mathcal{F}_{t,
T}^B}[Y_{t+\delta}] + \sigma_t Z_t+ \bar{\sigma}_t
E^{\mathcal{F}_{t, T}^B}[Z_{t+\delta}] + \rho_t) dt$ & \vspace{2mm}
\\ && $+ \kappa_t Y_t d \overleftarrow{B}_t - Z_t d{W}_t, $
& $ t\in[t_0, T];$ \vspace{2mm}
\\
$Y_t$ &=& $\xi_t, $ & $t\in[T, T+\delta];$ \vspace{2mm}
\\
$Z_t$ &=& $\eta_t, $ & $t\in[T, T+\delta]$
\end{tabular}\right.
\end{equation}
and the delayed DSDE
\begin{equation}\label{eq duality sde}
\left\{
\begin{tabular}{rlll}
$dX_s$ &=& $(\mu_s X_s +\bar{\mu}_{s-\delta}  X_{s-\delta})ds +
\kappa_s X_s d\overleftarrow{B}_s + (\sigma_s X_s
+\bar{\sigma}_{s-\delta} X_{s-\delta}) d{W}_s, $ & $ s\in[t, T];$
\vspace{2mm}
\\
$X_t$ &=& $1, $ & $$  \vspace{2mm}
\\
$X_s$ &=& $0, $ & $s\in[t-\delta, t),$
\end{tabular}\right.
\end{equation}
where we suppose that $t_0 \geq \delta >0$ are fixed constants,
$(\xi, \eta) \in S_{\mathcal{G}}^2(T, T+\delta; \mathbb{R})\times
L_{\mathcal{G}}^2(T, T+\delta; \mathbb{R}^d)$ with $\xi_T\in
L^2({\mathcal{F}_{T, T}^B}; \mathbb{R})$, $\mu_t$, $\bar{\mu}_t\in
L_{\mathcal{F}_{t, T}^B}^2(t_0-\delta, T+\delta; \mathbb{R})$,
$\sigma_t, \bar{\sigma}_t \in L_{\mathcal{F}_{t, T}^B}^2(t_0-\delta,
T+\delta; \mathbb{R}^d)$, $\kappa_t \in L_{\mathcal{F}_{t,
T}^B}^2(t_0, T; \mathbb{R}^l)$, $\rho_t \in L_{\mathcal{F}_{t,
T}^B}^2(t_0, T; \mathbb{R})$, and $\mu$, $\bar{\mu}$,
$\sigma$,$\bar{\sigma}$, $\kappa$ are uniformly bounded. Then by
Theorem \ref{thm existence uniqueness}, (\ref{eq duality bdsde}) has
a unique solution.

\begin{proposition}\label{prop duality}
Let $(Y, Z) \in S_\mathcal{G}^2(t_0, T+\delta; \mathbb{R})\times
L_\mathcal{G}^2(t_0, T+\delta; \mathbb{R}^d)$ be the unique solution
of ABDSDE (\ref{eq duality bdsde}). Then for $t\in[t_0, T]$,
$Z_t\equiv 0$, and $Y_t$ is $\mathcal{F}_{t, T}^B$-progressively
measurable.
\end{proposition}
\begin{proof}
First we show that $Y_t$ is $\mathcal{F}_{t, T}^B$-progressively
measurable. For this we introduce the following auxiliary equation:
\begin{equation}\label{eq duality bdsde auxiliary}
\left\{
\begin{tabular}{rlll}
$-dY_t^\prime$ &=& $E^{\mathcal{F}_{t, T}^B}[(\mu_t+\kappa_t^2)
Y_t^\prime+\bar{\mu}_t Y_{t+\delta}^\prime + \sigma_t Z_t^\prime +
\bar{\sigma}_t Z_{t+\delta}^\prime + \rho_t] dt$ & \vspace{2mm}
\\ && $+ \kappa_t E^{\mathcal{F}_{t, T}^B} [Y_t^\prime] d \overleftarrow{B}_t - Z_t^\prime d{W}_t, $
& $ t\in[t_0, T];$ \vspace{2mm}
\\
$Y_t^\prime$ &=& $\xi_t, $ & $t\in[T, T+\delta];$ \vspace{2mm}
\\
$Z_t^\prime$ &=& $\eta_t, $ & $t\in[T, T+\delta]$
\end{tabular}\right.
\end{equation}
which has a unique solution according to Theorem \ref{thm existence
uniqueness}.

In fact, it is obvious that
$$Y_t^\prime = E^{\mathcal{G}_t}[\Theta^\prime],$$ where $$\Theta^\prime:=\xi_T + \int_t^T
E^{\mathcal{F}_{s, T}^B}[(\mu_s+\kappa_s^2) Y_s^\prime+\bar{\mu}_s
Y_{s+\delta}^\prime + \sigma_s Z_s^\prime + \bar{\sigma}_s
Z_{s+\delta}^\prime + \rho_s] ds +\int_t^T \kappa_s
E^{\mathcal{F}_{s, T}^B} [Y_s^\prime] d\overleftarrow{B}_s$$ is
$\mathcal{F}_{t, T}^B$ measurable thanks to the fact that $\xi_T\in
L^2({\mathcal{F}_{T, T}^B}; \mathbb{R})$ and $\mu_t$, $\bar{\mu}_t$,
$\sigma_t$, $\bar{\sigma}_t$, $\kappa_t$, $\rho_t$ are all
$\mathcal{F}_{t, T}^B$-progressively measurable. Note that
$\mathcal{F}_{0, t}^W \vee \mathcal{F}_{T, T+K}^B$ is independent of
$\mathcal{F}_{t, T}^B \vee \sigma(\Theta^\prime)$, hence we know
$$Y_t^\prime = E^{\mathcal{F}_{t,
T}^B} [\Theta^\prime].$$ Thus obviously $Z_t^\prime \equiv 0$, and
moreover, $E^{\mathcal{F}_{t, T}^B}[Y_t^\prime]=Y_t^\prime$,
$E^{\mathcal{F}_{t, T}^B}[Z_t^\prime]=Z_t^\prime$. Then by Comparing
the anticipated BDSDE (\ref{eq duality bdsde}) with (\ref{eq duality
bdsde auxiliary}), together with the uniqueness of their solutions,
we immediately get the desired conclusion.
\end{proof}

The next is our main result.

\begin{theorem}
For any $(\xi, \eta) \in S_{\mathcal{G}}^2(T, T+\delta;
\mathbb{R})\times L_{\mathcal{G}}^2(T, T+\delta; \mathbb{R}^d)$ with
$\xi_T\in L^2({\mathcal{F}_{T, T}^B}; \mathbb{R})$, the solution
$Y_\cdot$ of the anticipated BDSDE (\ref{eq duality bdsde}) can be
given by
\begin{equation*}
Y_t =  E^{\mathcal{F}_{t, T}^B} [X_T \xi_T + \int_t^T \rho_s X_s ds]
+ E^{\mathcal{F}_{t, T}^B} [\int_{T}^{T+\delta}
(\bar{\mu}_{s-\delta} X_{s-\delta} E^{\mathcal{F}_{s-\delta,
T}^B}[\xi_{s}] +\bar{\sigma}_{s-\delta} X_{s-\delta}
E^{\mathcal{F}_{s-\delta, T}^B}[\eta_{s}]) ds],
\end{equation*}
where $X_\cdot$ is the unique solution of delayed DSDE (\ref{eq
duality sde}).
\end{theorem}

\begin{proof}
We first show that DSDE (\ref{eq duality sde}) has a unique
solution. In fact, when $s\in [t, t+\delta]$,
\begin{equation}\label{eq duality sde 1}
\left\{
\begin{tabular}{rlll}
$dX_s$ &=& $\mu_s X_s ds + \kappa_s X_s d\overleftarrow{B}_s +
\sigma_s X_s d{W}_s, $ & $ s\in[t, t+\delta];$ \vspace{2mm}
\\
$X_t$ &=& $1. $ & $$
\end{tabular}\right.
\end{equation}
Then we can easily obtain a unique solution $\varsigma_\cdot^1$ for
(\ref{eq duality sde 1}). When $s\in [t+\delta, t+2\delta]$,
\begin{equation}\label{eq duality sde 2}
\left\{
\begin{tabular}{rlll}
$dX_s$ &=& $ (\mu_s X_s +
\bar{\mu}_{s-\delta}\varsigma_{s-\delta}^1)ds + \kappa_s X_s
d\overleftarrow{B}_s + (\sigma_s X_s +
\bar{\sigma}_{s-\delta}\varsigma_{s-\delta}^1) d{W}_s, $ & $
s\in[t+\delta, t+2\delta];$ \vspace{2mm}
\\
$X_{t+\delta}$ &=& $\varsigma_{t+\delta}^1. $ & $$
\end{tabular}\right.
\end{equation}
Then we can easily obtain a unique solution $\varsigma_\cdot^2$ for
(\ref{eq duality sde 2}). Similarly ,we can consider all the other
cases when $t\in [t+2\delta, t+3\delta]$, $[t+3\delta, t+4\delta]$,
$\cdots$, $[t+ [\frac{T-t}{\delta}]\delta, T]$. Thus DSDE (\ref{eq
duality sde}) has a unique solution $X \in
\mathcal{S}_{\tilde{\mathcal{G}}}^2(t-\delta, T; \mathbb{R})$ where
$\tilde{\mathcal{G}}_t:= \mathcal{F}_t^W \vee \mathcal{F}_t^B$.

Applying It\^{o}'s formula to $X_sY_s$, according to Proposition
\ref{prop duality}, we have
\begin{equation}\label{eq duality 1}
\begin{tabular}{rlll}
& $X_T Y_T - X_t Y_t - \int_t^T (X_s Z_s + \sigma_s X_s Y_s +
\bar{\sigma}_{s-\delta} X_{s-\delta} Y_s) d{W}_s $ \vspace{5mm} \\
= & $\int_t^T (\bar{\mu}_{s-\delta} X_{s-\delta} Y_s -\bar{\mu}_s
X_s E^{\mathcal{F}_{s, T}^B}[Y_{s+\delta}] + \bar{\sigma}_{s-\delta}
X_{s-\delta} Z_s -\bar{\sigma}_s X_s E^{\mathcal{F}_{s,
T}^B}[Z_{s+\delta}]- \rho_s X_s) ds$ \vspace{5mm} \\
= & $\int_t^{T-\delta} (\bar{\mu}_{s-\delta} X_{s-\delta} Y_s
-\bar{\mu}_s X_s Y_{s+\delta}) ds$ \vspace{5mm} \\
& $+ \int_{T-\delta}^T (\bar{\mu}_{s-\delta} X_{s-\delta} Y_s
-\bar{\mu}_s X_s E^{\mathcal{F}_{s, T}^B}[\xi_{s+\delta}]
-\bar{\sigma}_s X_s E^{\mathcal{F}_{s, T}^B}[\eta_{s+\delta}]) ds -
\int_t^{T} \rho_s X_s ds.$
\end{tabular}
\end{equation}
Write $\Delta=\int_t^{T-\delta} (\bar{\mu}_{s-\delta} X_{s-\delta}
Y_s -\bar{\mu}_s X_s Y_{s+\delta})ds$, then
\begin{equation}\label{eq duality 2}
\begin{tabular}{rlll}
$\Delta =$ & $\int_t^{T-\delta} (\bar{\mu}_{s-\delta} X_{s-\delta}
Y_s -\bar{\mu}_s X_s Y_{s+\delta}) ds= \int_t^{T-\delta}
\bar{\mu}_{s-\delta} X_{s-\delta} Y_s ds -
\int_{t+\delta}^{T} \bar{\mu}_{s-\delta} X_{s-\delta} Y_s ds$ \vspace{5mm} \\
= & $\int_t^{t+\delta} \bar{\mu}_{s-\delta} X_{s-\delta} Y_s ds -
\int_{T-\delta}^{T} \bar{\mu}_{s-\delta} X_{s-\delta} Y_s ds =  -
\int_{T-\delta}^{T} \bar{\mu}_{s-\delta} X_{s-\delta} Y_s ds,$
\end{tabular}
\end{equation}
and the last equality is due to the fact that $X_s=0,\ s\in
[t-\delta, t)$.
\\
Combining (\ref{eq duality 1}) and (\ref{eq duality 2}), we have
\begin{align*}
& X_T Y_T - X_t Y_t - \int_t^T (X_s Z_s + \sigma_s X_s Y_s +
\bar{\sigma}_{s-\delta} X_{s-\delta} Y_s) d{W}_s \\
= & - \int_{T-\delta}^T (\bar{\mu}_s X_s E^{\mathcal{F}_{s,
T}^B}[\xi_{s+\delta}] +\bar{\sigma}_s X_s E^{\mathcal{F}_{s,
T}^B}[\eta_{s+\delta}]) ds - \int_t^{T} \rho_s X_s ds.
\end{align*}
Take conditional expectation with respect to
$\tilde{\tilde{\mathcal{G}}}_t:=\mathcal{F}_t^W \vee
\mathcal{F}_{T}^B$ on both sides, then
\begin{align*}
X_t Y_t  = & E^{\tilde{\tilde{\mathcal{G}}}_t} [X_T Y_T + \int_t^T
\rho_s X_s ds] + E^{\tilde{\tilde{\mathcal{G}}}_t}
[\int_{T-\delta}^T (\bar{\mu}_s X_s E^{\mathcal{F}_{s,
T}^B}[\xi_{s+\delta}] +\bar{\sigma}_s X_s E^{\mathcal{F}_{s,
T}^B}[\eta_{s+\delta}]) ds],
\end{align*}
which implies the desired result when noting that $X_t=1$ and
$Y_t\in \mathcal{F}_{t, T}^B$.
\end{proof}

\section{Conclusion and future work}

In this paper, we have established the existence/uniqueness theorem
and the comparison theorem for the anticipated BDSDEs. Moreover, as
an application, we studied a duality between the anticipated BDSDE
and delayed DSDE, where the BDSDE is of a special form, thus the
duality is somewhat limited. In fact for the general case, it should
be admitted that $\mathcal{G}_t$, which is not a filtration, not
increasing non decreasing, brings the main technical difficulty in
working with the duality problem. For the future work, I will go on
studying this topic and pay more attention to the applications of
such equations.



\section*{Acknowledgements}

The author is grateful to the editor and anonymous referees for
their helpful suggestions.

This work is supported by the Mathematical Tianyuan Foundation of
China (Grant No. 11126050), the Specialized Research Fund for the
Doctoral Program of Higher Education of China (Grant No.
20113207120002), and partially supported by the National Natural
Science Foundation of China (Grant No. 11101209).

\end{document}